\documentclass{amsart}
\usepackage{amsmath}
\usepackage{amsfonts}
\usepackage{amssymb}
\usepackage{epsfig}
\newtheorem{theorem}{Theorem}[section]

\newtheorem{lm}[theorem]{Lemma}
\newtheorem{tr}[theorem]{Theorem}
\newtheorem{cor}[theorem]{Corollary}

\newtheorem{rem}[theorem]{Remark}
\newtheorem{pr}[theorem]{Proposition}

\newtheorem{question}[theorem]{Question}

\newcommand{\cal}{\mathcal}
\begin{document}
\title{$GL(m|n)$-supermodules with good and Weyl filtrations}
\author{Alexandr N. Zubkov}
\address{Omsk State Pedagogical University, Chair of Mathematical Analysis, Algebra and Geometry, 644099 Omsk-99, Tuhachevskogo Embankment 14, Russia}
\email{a.zubkov@yahoo.com}
\begin{abstract}
The purpose of this paper is to prove necessary and sufficient criteria for a $GL(m|n)$-supermodule to have a good or Weryl filtration.
We also introduce the notion of a Steinberg supermodule analogous to the classical notion of Steinberg module. We prove that the Steinberg supermodule inherits some properties of the Steinberg module. Some new series of finite-dimensional tilting supermodules are found. 
\end{abstract}
\maketitle

\section*{Introduction}

Modules with costandard (good) or standard (Weyl) filtrations and tilting modules as well, play a crucial role in the representation theory of finite-dimensional algebras
and reductive algebraic groups over an algebraically closed field of positive characteristic. The concept of a tilting module can be adapted 
to any highest weight category $\cal{C}$ in the sense of \cite{cps}. It has been proven in \cite{cps} that if the partially ordered set of weights ({\it weight poset}) 
$\Lambda$ of such category $\cal{C}$ is finite, then 
the full subcategory consisting of all objects of finite length is equivalent to a category of finite-dimensional modules over
a quasi-hereditary algebra (a nice introduction to the theory of quasi-hereditary algebras and tilting modules over them
can be found in \cite{don0}; see also \cite{don2, ring}). In many cases, like Schur algebras or rational representations  of reductive algebraic groups, 
$\Lambda$ is not finite but any finitely generated ideal $\Gamma$ of the poset $\Lambda$ is. In that case, the full subcategory of finite objects belonging to $\Gamma$ is 
equivalent to a category of finite-dimensional modules over a quasi-hereditary algebra.

It was proven in \cite{zub1} that the category $GL(m|n) - \mathsf{Smod}$ of left rational supermodules
over the general linear supergroup $GL(m|n)$ is a highest weight category. However, the poset of weights $\Lambda$ of $GL(m|n)$ does not satisfy the condition that every finitely-generated ideal $\Gamma$ of $\Lambda$ is finite. In fact, for every weight $\lambda\in\Lambda$, the interval $\{\mu |\mu\leq\lambda \}$ is infinite. 

In \cite{zubmar} it was suggested how to overcome this obstacle. Since a highest weight category is an abelian category of finite type,  it can be regarded as a right 
comodule category over a coalgebra. Equivalently, it can
be viewed as a left discrete module category over a {\it pseudocompact} algebra. If $\Gamma$ 
is finitely generated (or finitely cogenerated, respectively) {\it good} (or {\it cogood}, respectively) ideal (see Definitions 3.9 and 3.17 of \cite{zubmar}), then 
the corresponding highest weight category is equivalent to a category
of discrete modules over an {\it ascending} (or {\it descending}, respectively) quasi-hereditary pseudocompact
algebra. The theory of objects with good ({\it decreasing costandard}) or Weyl ({\it increasing standard}) filtration 
has been developed in \S4 of \cite{zubmar}. The notion of a tilting object was introduced there and it was proven that
the category of discrete modules over an ascending quasi-hereditary peseudocompact algebra has enough indecomposable tilting objects. 
However, if $\Gamma$ is infinite, it is not clear if the tilting objects are finite or not. It is natural to look for a sufficient condition on the highest
weight category $\cal{C}$ that would guarantee that all of its indecomposable tilting objects are finite.

It makes sense to investigate first the category $GL(m|n)-\mathsf{Smod}$.
The costandard and standard objects in $GL(m|n)-\mathsf{Smod}$ are (up to a parity shift) the induced supermodules $H^0(\lambda)$ and the Weyl supermodules $V(\lambda)$, respectively 
(cf. \cite{zub1}). If $H^0(\lambda)$ is irreducible, then 
$H^0(\lambda)=V(\lambda)$ is tilting. 
Irreducible induced supermodules $H^0(\lambda)$ were characterized in \cite{marko}. Until recently, no other examples of tilting supermodules in $GL(m|n)-\mathsf{Smod}$ 
has been known, except in the case of the supergroup $GL(1|1)$ or in the case when $char K=0$. In both cases every indecomposable tilting supermodule is injective,  
projective and finite-dimensional; see \cite{bkn, zubmar}.  

In this article we present a series of finite-dimensional (indecomposable) tilting supermodules, which are not irreducible, using the concept of a Steinberg supermodule. 
We prove that the Steinberg supermodule inherits some properties of 
its counterpart, a Steinberg module $St_r$. For example, a Steinberg supermodule remains both projective and injective when regarded as a supermodule over the corresponding Frobenius kernel in $GL(m|n)$. 

It was proven in \cite{don3} (see also II.10.5.2 (1) of \cite{jan}) that, for every (dominant) weight $\lambda$, the tensor product
$St_r\otimes I(\lambda)^{[r]}$ is isomorphic to $I((p^r -1)\rho +p^r\lambda)$, where $I(\mu)$ is the injective envelope of
an irreducible module $L(\mu)$ of the highest weight $\mu$. We prove an analogous statement for Steinberg supermodules. More specifically, the tensor product of the $r$-th Steinberg supermodule with the $r$-th {\it even} Frobenius twist of an indecomposable injective module over 
$GL(m|n)_{res}$ is both injective and indecomposable. The most difficult part of the proof is to show that this tensor product 
is injective. To prove it, we use the aforementioned fact that the Steinberg supermodule is injective over the corresponding Frobenius kernel and some spectral sequence arguments 
(see Lemma \ref{strangelemma} and Theorem \ref{tensorproductisinjective}).
Our series of tilting supermodules is obtained in a similar way; they are tensor products of Steinberg supermodules with even Frobenius twists of 
tilting modules over $GL(m|n)_{res}$. This generalizes Proposition (2.1) from \cite{don2} for general linear supergroups.
 
The paper is organized as follows. In the first seven sections and in the ninth section we give all necessary definitions, notations and derive auxiliary  results. In the eighth section we prove two criteria for a $GL(m|n)$-supermodule to have a decreasing good or increasing Weyl filtration (both possibly infinite). The results of this section are interesting on their own and they are used in the last section to describe completely all cases when a symmetric or an exterior (super)power of the standard $GL(m|n)$-supermodule $W$ have a good or Weyl filtration. In the tenth section we introduce the notion of a Steinberg supermodule and prove some of its important properties mentioned earlier. 

\section{Hopf superalgebras}

We follow definitions and notations from \cite{maszub}. For the convenience of the reader we will recall some of them here.
Let $A$ be an (associative) superalgebra. The category of left (or right, respectively) $A$-supermodules (with ungraded morphisms) is denoted by $_{A}\mathsf{SMod}$
(or by $\mathsf{SMod}_A$, respectively). If $A$ is a supercoalgebra, then the category of left (or right, respectively) $A$-supercomodules (with ungraded morphisms) is denoted by $^{A}\mathsf{SMod}$ (or by $\mathsf{SMod}^A$, respectively). These categories are not abelian, but their {\it underlying even categories} $_{A}\mathsf{Smod}, \mathsf{Smod}_A, 
^{A}\mathsf{Smod},$ $\mathsf{Smod}^A$, consisting of the same objects but considered only with even morphisms, are. 

Let $H$ be a Hopf superalgebra. A Hopf subsuperalgebra $R$ of $H$ is called {\it right normal} if $\sum (-1)^{|r||h_1|}s_H(h_1)rh_2\in R$ for all $r\in R$ and $h\in H$
(cf. \cite{amas}). Symmetrically, $R$ is called {\it left normal} if $\sum (-1)^{|r||h_2|}h_1 rs_H(h_2)\in R$ for all $r\in R$ and $h\in H$. 
If $R$ is both left and right normal, then $R$ is called just {\it normal}. 
If $H$ is supercocommutative and the antipode $s_H$ is bijective, then the right normality is equivalent to the left normality.
 
Let $M$ be a left $H$-supermodule. Denote by $M_R$ a subsuperspace of $M$ that is generated by the elements
$rm$ for $m\in M$ and $r\in R^+ =\ker\epsilon_R$.

\begin{lm}\label{normalhopf}
If $R$ is left normal in $H$, then $M_R$ is a $H$-subsupermodule.
\end{lm}
\begin{proof}
For every $r\in R^+$ we have 
$$(hr)m=\sum(\sum (-1)^{|r||h_2|} h_1 r s_H(h_2)) h_3 m\in M_R .$$ 
\end{proof}

Let $R$ be a finite-dimensional Hopf superalgebra. The dual superspace $R^*$ has a natural structure of Hopf superalgebra given by
\[(\phi\psi)(r)=\sum (-1)^{|\psi||r_1|}\phi(r_1)\psi(r_2), \mbox{ where } \phi, \psi\in R^*, r\in R \mbox { and } \Delta_R(r)=\sum r_1\otimes r_2;\]
$\Delta_{R^*}(\phi)=\sum\phi_1\otimes\phi_2$ whenever 
$\phi(rs)=\sum(-1)^{|\phi_2||r|}\phi_1(r)\phi_2(s)$ for every $r, s\in R$;
$s_{R^*}(\phi)(r)=\phi(s_R(r))$, and $\epsilon_{R^*}(\phi)=\phi(1_R)$.
\begin{lm}\label{self-duality}
The functor $\Phi : R\to R^*$ is a self-duality on the category of finite-dimensional Hopf superalgebras.  
\end{lm}
\begin{proof}
Follow the arguments in I.8(1) of \cite{jan}.
\end{proof}
Additionally, if $M$ is a left $R$-supermodule, then $M$ is a right $R^*$-supercomodule via the linear map
$M\to Hom_K(R, M)\simeq M\otimes R^*$ which maps $m\in M$ to $l_m : R\to M$ given by $l_m(r)=(-1)^{|m||r|}rm$ for $r\in R$.
Symmetrically, if $M$ is a right $R$-supercomodule, then $M$ is a left $R^*$-supermodule by 
$\phi m=\sum(-1)^{|\phi||m_1|}\phi(r_2)m_1$, where $\tau_M(m)=\sum m_1\otimes r_2$, $\phi\in R^*$ and $m\in M$.
Furthemore, all of the above statements remain valid after replacing the right coaction by the left action and vice-versa.

The proof of the following lemma is easy and is left for the reader.
\begin{lm}\label{equivalence}
There is an equivalence of categories $\mathsf{SMod}^R\simeq {}_{R^*}\mathsf{SMod}$ that preserves parities of morphisms.
Symmetrically, ${}^R\mathsf{SMod}\simeq \mathsf{SMod}_{R^*}$.
\end{lm}
In particular, the category $\mathsf{SMod}^R$ has enough projective objects. Thus for any $M, N\in 
\mathsf{SMod}^R$ we have 
$Ext^n_{\mathsf{SMod}^R}(M, N)\simeq H^n(Hom_{\mathsf{SMod}^R}(P^{\bullet}_M, N))$, where $P^{\bullet}_M\to M\to 0$ is a projective resolution of $M$ 
(cf. Proposition 8.2 in XX of \cite{lang}). Moreover, the resolution $P^{\bullet}_M$ can be chosen in
$\mathsf{Smod}^R$.

\section{Algebraic supergroups and their linear representations}

Let $\mathsf{SAlg}_K$ denote the category of supercommutative $K$-superalgebras with even morphisms. Let $B$ be a supercommutative Hopf superalgebra. Then the representable functor $A\to Hom_{\mathsf{SAlg}_K}(B, A)$ from $\mathsf{SAlg}_K$ to the category of sets is a natural group 
functor. It is denoted by $G=SSp \ B$ and called an {\it affine supergroup}. If $B$ is finitely generated, then $G$ is called
an {\it algebraic supergroup}. Any (closed) subsupergroup $H$ of $G$ is uniquely defined by a Hopf superideal $I_H$ of $B$
such that an element $g\in G(A)$ belongs to $H(A)$ if and only if $g(I_H)=0$. For example, the {\it largest even subsupergroup} 
$G_{ev}$ of $G$ corresponds to the ideal $BB_1$. The restriction of $G_{ev}$ to the full subcategory of commutative $K$-algebras
is denoted by $G_{res}$. 

Let $W$ be a finite-dimensional superspace. The group functor $A\to \mathsf{End}_A(W\otimes A)_{0}^*$ is an algebraic supergroup. It is called a {\it general linear supergroup} and denoted by $GL(W)$. If $\dim W_0=m, \dim W_1=n$, then $GL(W)$ is also denoted by $GL(m|n)$. 

Fix a homogeneous basis of $W$, say consisting of $w_i$ for $1\leq i\leq m+n$, where $|w_i|=0$ provided $1\leq i\leq m$, and $|w_i|=1$ otherwise.
It is easy to see that $K[GL(m|n)]$ is a localization of $K[c_{ij}|1\leq i, j\leq m+n]$, where $|c_{ij}|=|w_i|+|w_j|$, by some element $d$.
More precisely, the generic matrix $C=(c_{ij})_{1\leq i, j\leq m+n}$ has a block form 
$$\left(\begin{array}{cc}
C_{00} & C_{01} \\
C_{10} & C_{11}
\end{array}\right),
$$
where the $m\times m$ and $n\times n$ blocks $C_{00}$ and $C_{11}$ are even, and the $m\times n$ and $n\times m$ blocks
$C_{01}$ and $C_{10}$ are odd. Then  
$$\Delta_{GL(m|n)}(c_{ij})=\sum_{1\leq k\leq m+n} c_{ik}\otimes c_{kj}, \ \epsilon_{GL(m|n)}(c_{ij})=\delta_{ij},$$
and $d=\det(C_{00})\det(C_{11})$. 

The element $Ber(C)=\det(C_{00}-C_{01}C_{11}^{-1}C_{10})\det(C_{11})^{-1}$ is called the {\it Berezinian}.
It is a group-like element of the Hopf superalgebra $K[GL(m|n)]$ (cf. \cite{ber}). It is easy to show that
$K[GL(m|n)]=K[c_{ij}|1\leq i, j\leq m+n]_{Ber(C)}$.

By definition, the category of left (or right, respectively) $G$-supermodules coincides with the category of right (or left, respectively) $K[G]$-supercomodules. 
Denote them by $G-\mathsf{SMod}$ and $\mathsf{SMod}-G$ respectively. The corresponding even underlying categories are denoted by
$G-\mathsf{Smod}$ and $\mathsf{Smod}-G$, respectively. For example, the right supercomodule structure of $GL(W)$-supermodule $W$ is defined by
$$\tau_W(w_i)=\sum_{1\leq j\leq m+n} w_j\otimes c_{ji}.$$

There is an endofunctor $M\to \Pi M$, called the {\it parity shift}, in all above categories such that $\Pi M$ coincides with $M$ as a $K[G]$-comodule and 
$(\Pi M)_i=M_{i+1}$ for $i\in\mathbb{Z}_2$.  

Every $G_{res}$-module can be regarded as a {\it purely even} $G_{ev}$-supermodule. Furthemore, $G_{res}-\mathsf{mod}$
is a full subcategory of both $G_{ev}-\mathsf{Smod}$ and $G_{ev}-\mathsf{SMod}$ and
$Ext^n_{G_{res}}(M, N)=Ext^n_{G_{ev}}(M, N)$ for every $M, N\in G_{res}-\mathsf{mod}$ and $n\geq 0$.

There is a one-to-one correspondence between $G$-supermodule structures on a finite-dimensional superspace $M$ and {\it linear representations}
$G\to GL(M)$ (cf. \cite{jan, zub1}). If $\tau_M(m)=\sum m_1\otimes f_2\in M\otimes K[G]$, then $g\in G(A)$ acts on $M\otimes A$ by the even 
$A$-linear automorphism $\tau(g)(m\otimes 1)=\sum m_1\otimes g(f_2)$. 
In other words, $M$ is a $G$-supermodule if and only if the group functor $G$ acts on the functor $M_a=M\otimes ?$ in such a way that for every $A\in\mathsf{SAlg}_K$ 
the group $G(A)$ acts on $M_a(A)=M\otimes A$ by even $A$-linear automorphisms. 

A subsupergroup $H$ is called {\it (faithfully) exact} in $G$ if the induction functor $ind^G_H$ is (faithfully) exact.
The exactness of $H$ is equivalent to the condition that the restriction functor $M\to M|_H$ takes injectives to injectives
(cf. \cite{jan, zub1}). Besides, $H$ is faithfully exact if and only if $G/H$ is an affine superscheme (see Theorem 5.2 of \cite{zub2}). 

In what follows, we use the following criterion for $G/H$ to be an affine superscheme (cf. Corollary 8.15 of \cite{maszub}).
A superscheme $G/H$ is affine if and only if $G_{ev}/H_{ev}$ is affine if and only if $G_{res}/H_{res}$ is affine. 
For example, if $H$ is a normal or finite subsupergroup of $G$, then $G/H$ is affine, hence $H$ is faithfully exact in $G$.

Let $B$ denote a {\it standard Borel subsupergroup} of $G=GL(W)$ that consists of all lower triangular matrices. Let $V$ be its largest unipotent subsupergroup that consists of all lower triangular matrices with units on their diagonals. 

Let $T$ be a maximal torus of $B$ consisting of all diagonal matrices and let $X(T)$ be its {\it character group}.
The group $X(T)$ can be naturally identified with $\mathbb{Z}^{m+n}$ so that an element $\lambda\in X(T)$ has a form
$(\lambda_1, \ldots , \lambda_{m+n})$, where $\lambda_i\in\mathbb{Z}$ for $1\leq i\leq m+n$. Let $\epsilon_i$ denote a character
$$(0, \ldots, \underbrace{1}_{i-\mbox{th place}}, \ldots, 0) \mbox{ for } 1\leq i\leq m+n .$$

Let $B^{opp}$ denote the transpose of $B$,
that is $B^{opp}$ is a Borel subsupergroup consisting of all upper triangulat matrices. Then the transpose $V^{opp}$ of $V$ is the largest unipotent
subsupergroup of $B^{opp}$.

Any simple $B$-supermodule has the form $K_{\lambda}$ or $\Pi K_{\lambda}$ for some $\lambda\in X(T)$, where the $K[B]$-comodule structure
of $K_{\lambda}$ is given by $1\mapsto 1\otimes\prod_{1\leq i\leq m+n} c_{ii}^{\lambda_i}$. 
The {\it induced supermodule} $ind^G_B \Pi ^a K_{\lambda}$ is denoted by $H^0(\lambda^a)$ for $a=0, 1$. Observe that 
$H^0(\lambda^a)=\Pi^a H^0(\lambda)\neq 0$ if and only if $\lambda$ is a {\it dominant} weight, that is $\lambda_1\geq\ldots\geq\lambda_m, \lambda_{m+1}\geq\ldots\geq\lambda_{m+n}$. If it is the case, then $H^0(\lambda^a)$ has a simple socle $L(\lambda^a)=\Pi^a L(\lambda)$ and any irreducible $GL(W)$-supermodule is isomorphic to the socle of some
$H^0(\lambda^a)$. The set of dominant weights is denoted by $X(T)^+$.
 
As it has been shown in
\cite{zub1}, the even category of $GL(W)$-supermodules is a highest weight category with $H^0(\lambda^a)$ as costandard objects, subject to a {\it Bruhat-Tits order} such that $\mu\leq\lambda$ if and only if 
$$\sum_{1\leq i\leq k}\mu_i\leq \sum_{1\leq i\leq k}\lambda_i \mbox{ for } 1\leq i\leq m+n-1, \mbox{ and } \sum_{1\leq i\leq m+n}\mu_i =\sum_{1\leq i\leq m+n}\lambda_i$$
if and  only if
$$\lambda-\mu\in \sum_{1\leq i < j\leq m+n}\mathbb{Z}_{\geq 0}(\epsilon_i -\epsilon_j).$$
The standard object of the highest weight $\lambda$ in $GL(W)-\mathsf{Smod}$ is denoted by $V(\lambda)$. 

Let $\Gamma$ be finitely generated ideal in $X(T)^+$. We say that a $GL(W)$-supermodule $M$ belongs to $\Gamma$ whenever any composition factor of $M$
is isomorphic to some $L(\lambda^a)$ with $\lambda\in\Gamma$. For any $GL(W)$-supermodule $N$ there exists a unique maximal subsupermodule of $N$ which belongs to $\Gamma$. It is denoted by $O_{\Gamma}(N)$. Also, there exists a unique minimal subsupermodule $N'$ of $N$ such that $N/N'$ belongs to $\Gamma$. The subsupermodule $N'$ is denoted by $O^{\Gamma}(N)$.
For notions of (decreasing) good and (increasing) Weyl filtrations
we refer the reader to \S 4 of \cite{zubmar}. 

Let $H_{ev}(\lambda)$ and $V_{ev}(\lambda)$ denote costandard and standard objects in the highest weight category $GL(W)_{res}-\mathsf{mod}$, respectively (subject to the same partial order). Then the category $GL(W)_{ev}-\mathsf{Smod}$ is also a highest weight category with costandard objects
$\Pi^a H_{ev}(\lambda)$ and standard objects $\Pi^a V_{ev}(\lambda)$
for $a=0, 1$. 

\section{Superalgebras of distributions}

Let $G$ be an algebraic supergroup and $Dist(G)$ be the {\it superalgebra of distributions} of $G$. As a superspace $Dist(G)$ coincides with
$\bigcup_{k\geq 0} Dist_k(G)\subseteq K[G]^*$, where $Dist_k(G)=(K[G]/\mathfrak{m}^{k+1})^*$ and $\mathfrak{m}=\ker\epsilon_{K[G]}=K[G]^+$. $Dist(G)$ is a supercocommutative Hopf superalgebra with a bijective antipode (see \cite{jan, zub2} for more details). For example, the comultiplication $\Delta_{Dist(G)}$ maps an element $\phi\in Dist_k(G)$
to $\sum \phi_1\otimes\phi_2\in Dist_k(G)^{\otimes 2}$ if and only if
$$\phi(fh)=\sum (-1)^{|\phi_2||f|}\phi_1(f)\phi_2(h)$$ for every $f, h\in K[G]$.

If $H$ is a subsupergroup of $G$, then $Dist(H)$ is a Hopf subsuperalgebra of $Dist(G)$. More precisely,
$\phi\in Dist(G)$ belongs to $Dist(H)$ if and only if $\phi(I_H)=0$.
Furthemore, if $H$ is a normal subsupergroup of $G$, then  
for every $f\in I_N$ we have $\sum (-1)^{|f_1||f_2|} f_2\otimes s_G(f_1) f_3\in I_N\otimes K[G]$ (cf. p. 731 of \cite{zub2}).
Therefore $Dist(H)$ is normal in $Dist(G)$. If $G$ is {\it connected} ({\it pseudoconnected} in the terminology of \cite{zub2}; see also \S 3 of \cite{zubgrish}), then the converse statement is also valid (follow arguments in the proof of Lemma 5.1 of \cite{zubgrish}). 

\section{Unipotent and finite supergroups}
 
An algebraic supergroup $G$ is called {\it unipotent} if every simple $G$-supermodule $M$ is trivial, that is $M\simeq \Pi^a K$ for $a=0, 1$ (cf. \cite{zub3, amas}).

We refer the reader to \cite{drup} for more detailed introduction to the theory of finite or infinitesimal algebraic supergroups.
A finite supergroup $G$ is called {\it infinitesimal} if $K[G]^+$ is nilpotent. Therefore $Dist(G)=K[G]^*$ and, by Lemma \ref{equivalence}, the category
$G-\mathsf{SMod}$ is equivalent to the category ${}_{Dist(G)}\mathsf{SMod}$. 
The functor $F : G-\mathsf{SMod} \to \mathsf{SMod}_K$ such that $F(M)= M/M_{Dist(G)}$ is right exact. Denote its $n$-th left derived functor $L_n F(M)$ by $H_n(G, M)$. In other words, 
$H_n(G, M)=H_n(F(P^{\bullet}_M))$ for any projective resolution $P^{\bullet}_M$ of $M$. Theorem 7.1' from chapter XX of \cite{lang}
implies that $Ext^n_G(M, K)^*\simeq H_n(G, M)$ for every $n\geq 0$.

\begin{lm}\label{injandproj}
If $G$ is unipotent, then every injective $G$-supermodule is isomorphic to $K[G]^{\oplus I}\bigoplus \Pi K[G]^{\oplus J}$ for some (possibly infinite) index sets $I$ and $J$. 

If $G$ is also infinitesimal, then every projective $G$-supermodule is isomorphic to
$Dist(G)^{\oplus I}\bigoplus \Pi Dist(G)^{\oplus J}$. 
\end{lm}
\begin{proof}
Every injective $G$-supermodule is a direct sum of indecomposable injective hulls of simple supermodules. On the other hand, the socle of a $G$-supermodule $M$ is equal to $M^G$. Since $K[G]^G=K$, up to an isomorphism, $K[G]$ and $\Pi K[G]$ are unique indecomposable injectives. If $G$ is infinitesimal and $P$ is a projective $G$-supermodule, then $rad P=Dist(G)^+ P$.
Since $Dist(G)/Dist(G)^+=K$, the claim follows. 
\end{proof}
\begin{lm}\label{injandproj2}
If $G$ is unipotent, then a $G$-supermodule $M$ is injective if and only if $H^1(G, M)=0$. If $G$ is also infinitesimal, then a $G$-supermodule $M$ is projective if and only if $H_1(G, M)=0$. Moreover, if $M$ is injective (or projective, respectively), then
$H^n(G, M)=0$ (or $H_n(G, M)=0$, respectively) for every $n\geq 1$.
\end{lm}
\begin{proof}
$Ext^1_G(K, M)=0$ implies $Ext^1_G(N, M)=0$ for every finite-dimensional $G$-supermodule $N$.
Following arguments in the proof of Proposition 3.4 in \cite{zub1}, we obtain that $Ext^1_G(N, M)=0$ for every $N$.
The proof of the second statement is similar.
\end{proof}
\begin{lm}\label{reduction}
Let $N\unlhd H$ and $H$ be unipotent. Then an $H$-supermodule $M$ is injective if and only if both the $H/N$-supermodule $M^N$ and the $N$-supermodule $M|_N$ are injective. 
\end{lm}
\begin{proof}
If $M|_N$ is an injective $N$-supermodule, then the spectral sequence
(cf. Proposition 3.1(3) of \cite{zubscal})
$$
E_2^{n, m}=H^n(H/N, H^m(N, M))\Rightarrow H^{n+m}(H, M)
$$
degenerates, which yields $H^n(H/N, M^N)\simeq H^n(H, M)$ for every $n\geq 1$. Lemma \ref{injandproj2} implies that 
$M$ is injective if and only if $M^N$ is an injective $H/N$-supermodule. On the other hand, if $M$ is injective, then
so is $M|_N$.
\end{proof}
Every superalgebra $A$ is a $\mathbb{Z}_2$-module, where the generator of $\mathbb{Z}_2$ acts on $A$ as $a\to (-1)^{|a|}a$ for $a\in A$. 
The {\it semi-direct product} algebra $A\rtimes\mathbb{Z}_2$ is isomorphic to $A_0\bigoplus A_1$, where each component
$A_{i}$ coincides with $A$ as a vector space for $i\in\mathbb{Z}_2$. Besides, $a_i a'_j=(-1)^{i|a'|}(aa')_{i+j}$ for $a_i\in A_i, a'_j\in A_j$ and $i, j\in\mathbb{Z}_2$. 
Additionally, if $A$ is a (not necessarily supercommutative) Hopf superalgebra, then
$A\rtimes\mathbb{Z}_2$ is a Hopf algebra with the comultiplication
$$\Delta(a_i)=\sum (a_1)_{i+|a_2|}\otimes (a_2)_i ,$$
the counit $\epsilon(a_i)=\epsilon_A(a)$ and the antipode $s(a_i)=(-1)^{(i+|a|)|a|}s_A(a)_{i+|a|}$.
\begin{lm}\label{projandinjoverinfinitesimal}
Let $H$ be an infinitesimal supergroup. Then a $H$-supermodule $M$ is injective if and only if $M$ is projective.
\end{lm}
\begin{proof}
By Lemma \ref{equivalence}, $H-\mathsf{SMod}\simeq _{Dist(H)}\mathsf{SMod}$. 
Since $Dist(H)\rtimes\mathbb{Z}_2$ is Frobenius by the main theorem from \cite{larsweed} and $_{Dist(H)}\mathsf{SMod}\simeq _{Dist(H)\rtimes\mathbb{Z}_2}\mathsf{Mod}$ by Lemma 7.6 from \cite{zubmar}, the statement follows.
\end{proof}

\section{Frobenius kernels}

Assume that $K$ is a perfect field and $V$ is a $K$-superspace. Denote by $V^{(r)}$ the superspace that coincides with $V$ as a $\mathbb{Z}_2$-graded abelian group, 
but on which each $a\in K$ acts as $a^{p^{-r}}$ does on $V$. 
If $A$ is a (not necessarily (super)commutative) (super)algebra, then $A^{(r)}$ is also a (super)algebra with respect to the same multiplication. Moreover, if $A$ is a Hopf (super)algebra, then $A^{(r)}$ is a Hopf (super)algebra with $\Delta_{A^{(r)}}=\Delta_A, s_{A^{(r)}}=s_A$ and $\epsilon_{A^{(r)}}=\epsilon_A^{p^r}$. From now on we,
assume that $K$ is perfect unless stated otherwise.

Let $G$ be an algebraic supergroup. Assume that $G$ is {\it reduced}, that is
$G_{res}$ is an reduced algebraic group. For every $r\geq 1$ there is an exact sequence 
$$1\to G_r\to G\to G_{ev}^{(r)}\to 1,$$
where the epimorphism $G\to G_{ev}^{(r)}$ is induced by the Hopf superalgebra embedding $K[G_{ev}^{(r)}]=K[G_{ev}]^{(r)}\to K[G]$ 
given by the rule $\overline{f}\to f^{p^r}$, where $\overline{f}$ is the residue class of $f\in K[G]_0$ in $K[G_{ev}]=K[G]_0/K[G]^2_1$
(see \S 8 of \cite{zub2}). The subsupergroup $G_r$ is called the $r$-{\it th Frobenius kernel} of $G$. 
Since $I_{G_r}=\sum_{f\in K[G]_0} K[G]f^{p^r}$, each $G_r$ is a (finite) infinitesimal subsupergroup in $G$.
Lemma \ref{equivalence} shows that the category of $G_r$-supermodules is naturally equivalent to the category of $Dist(G_r)$-supermodules. 

Let $G=GL(W)$. A simple $G_r$-supermodule is isomorphic to the top $L_r(\lambda)$ of $Z_r(\lambda)=Dist(G_r)\otimes_{Dist(B_r^{opp})} K_{\lambda}$
or to its parity shift $\Pi L_r(\lambda)$ (cf. Theorem 5.4 of \cite{kuj}). Moreover, $L_r(\lambda)\simeq L_r(\mu)$ if and only if
$\lambda-\mu\in p^r X(T)$. 

A simple $G_r$-supermodule $L_r(\lambda)$ is also isomorphic to the socle of $Z'_r(\lambda)=ind^{G_r}_{B_r} K_{\lambda}$. 
In fact, there is a natural superscheme isomorphism $B_r\times V^{opp}_r\to G_r$ (see Lemma 14.1 of \cite{zub4}). 
By Lemma 8.2 and Remark 8.3 of \cite{zub4}, $Z'_r(\lambda)|_{B^{opp}_r}\simeq K_{\lambda}\otimes K[V^{opp}_r]$, where $T$ acts on $K[V^{opp}_r]$ by conjugations and $V^{opp}_r$ acts trivially on $K_{\lambda}$. Thus $Z'_r (\lambda)^{V^{opp}_r}\simeq
K_{\lambda}$ as a $T_r$-(super)module, hence the socle of $Z'_r (\lambda)$ is irreducible.
Let $L'_r(\lambda)$ denotes the socle of $Z'_r (\lambda)$. Since $L'_r(\lambda)^{V^{opp}_r}\simeq K_{\lambda}$
, (co)Frobenius reciprocity law implies
$$Hom_{G_r}(Z_r(\lambda), L'_r(\lambda))\simeq Hom_{B^{opp}_r}(K_{\lambda}, L'_r(\lambda))
=K;$$
that is $L_r(\lambda)\simeq L'_r(\lambda)$.

Since $Dist(G)=Dist(G_{ev})Dist(G_r)$, one can mimic the proofs of Proposition 4.3 and Proposition 4.4 of
\cite{shuweiq} to show that
a simple $G$-supermodule $L(\lambda)$ of highest weight $\lambda\in X(T)^+$ is completely reducible as a $G_r$-supermodule. 
Furthemore, it is irreducible as $G_r$-supermodule whenever
$\lambda\in X_r(T)^+$, where $$X_r(T)^+=\{\lambda\in X(T) | 0\leq \lambda_i-\lambda_{i+1} < p^r \mbox{ for } 1\leq i\leq m+n-1 \mbox{ such that } i\neq m\}.$$   
It is easy to see that for every weight $\lambda\in X(T)$ there are $r\geq 1$ and $\mu\in X_r(T)^+$ such that $\lambda-\mu\in p^r X(T)$. 
Therefore, $L_r(\lambda)\simeq L(\mu)|_{G_r}$. 

\section{Frobenius twists}

Let $G$ be a reduced algebraic supergroup defined over $\mathbb{F}_p$. In other words, 
$K[G]=A\otimes_{\mathbb{F}_p} K$, where $A$ is a Hopf $\mathbb{F}_p$-superalgebra. Since 
$K[G_{ev}]\simeq A/AA_1\otimes _{\mathbb{F}_p} K$, $G_{ev}$ and $G_{res}$ are also defined over $\mathbb{F}_p$.
There is a natural Hopf superalgebra endomorphism $\pi_r : K[G]\to K[G]$ given by $f\otimes b\to f^{p^r}\otimes b$ for $f\in A$ and $b\in K$. 

Let $M$ be a left $G$-supermodule. Then one can define a new representation of $G$ on $M$, denoted by $M^{[r]}$, by
$\tau_{M^{[r]}}(m)=(id_M\otimes \pi_r)\tau_M$ (cf. 9.10, Part I of \cite{jan}). We call $M^{[r]}$ the $r$-th {\it Frobenius twist}
of $M$. 

It is evident that $G_r$ acts on $M^{[r]}$ trivially. Since $G/G_r\simeq G_{ev, r}$,  $M^{[r]}$ has the natural structure of a $G_{ev, r}$-supermodule via the Hopf (super)algebra isomorphism $K[G_{ev}]^{(r)}\simeq A_0^{p^r}\otimes_{\mathbb{F}_p} K\subseteq K[G]$ given by $\overline{f}\otimes b\to f^{p^r}\otimes b^{p^r}$, where $\overline{f}$ is the residue class of $f\in A_0$ in $A_0/A_1^2$ and $b\in K$.

Let $M$ be a $G_{res}$-module. It has been observed that $M$ is also a $G_{ev}$-supermodule. 
The supergroups $G_{ev}$ and $G_{ev}^{(r)}$ are isomorphic to each other (cf. Remark I.9.5 of \cite{jan}). More precisely, 
the corresponding Hopf superalgebra isomorphism $K[G_{ev}]\to K[G_{ev}^{(r)}]$ is defined as $h\otimes b\to h\otimes b^{p^{-r}}$ for 
$h\in A/AA_1$ and $b\in K$. In particular, $M$ has a natural $G$-supermodule structure via the epimorphism $G\to G_{ev}^{(r)}$
(see Remark \ref{inflation/restriction} later). This structure is defined as $$\tau'(m)=\sum m_1\otimes f_2^{p^r}\otimes b_2 ,$$
where $\tau_M(m)=\sum m_1\otimes\overline{f_2}\otimes b_2, f_2\in A_0$ and $b_2\in K$. We call such a $G$-supermodule the $r$-th {\it even Frobenius twist} of $M$ and 
by abuse of notation we also denote it by $M^{[r]}$. 
\begin{rem}\label{comparewithKujawa}
The concepts of the even Frobenius twist has been introduced in \cite{kuj}, where it has been called just the Frobenius twist (see also
\cite{drup}). Since the concept of the Frobenius twist makes sense for any supermodule, we decided to distinguish these (obviously different) constructions by their names. 
Observe also that if $M$ is a $G_{res}$-module, then the $r$-th Frobenius twist of $G_{ev}$-supermodule $M$ coincides with its $r$-th even Frobenius twist, restricted to $G_{ev}$.
\end{rem}
\begin{lm}\label{endomorphismring}
If $M$ is a $G_{res}$-module, then $End_G(M^{[r]})=End_{G_{res}}(M)$.
\end{lm}
\begin{proof}
Consider a basis $m_1, \ldots, m_k$ of $M$. Then 
$$\tau_M(m_i)=\sum_{1\leq j\leq k} m_j\otimes \overline{f_{ji}} \mbox { for } f_{ji}\in A_0\otimes_{\mathbb{F}_p} K,
$$
and
$$\tau_{M^{[r]}}(m_i)=\sum_{1\leq j\leq k} m_j\otimes \pi_r(f_{ji}) \mbox{ for } 1\leq i\leq k.$$
An endomorphism $\phi\in End_K(M)$ given by 
$$\phi(m_i)=\sum_{1\leq j\leq k}\phi_{ji}m_j \mbox{ for } 1\leq i\leq k$$
belongs to $End_G(M^{[r]})$ if and only if $F\Phi =\Phi F$, where
$$F=\left(\begin{array}{ccc}
\pi_r(f_{11}) & \ldots & \pi_r(f_{1k})\\
\vdots & \vdots & \vdots \\
\pi_r(f_{k1}) & \ldots & \pi_r(f_{kk})
\end{array}\right) \mbox{ and }\
\Phi=\left(\begin{array}{ccc}
\phi_{11} & \ldots & \phi_{1k}\\
\vdots & \vdots & \vdots \\
\phi_{k1} & \ldots & \phi_{kk}
\end{array}\right). 
$$
Since $\pi_r$ induces an algebra isomorphism $K[G_{res}]=K[G]_0/K[G]^2_1\to A_0^{p^r}\otimes_{\mathbb{F}_p} K$, 
the latter condition is equivalent to $\overline{F}\Phi=\Phi\overline{F}$, where
$$F=\left(\begin{array}{ccc}
\overline{f_{11}} & \ldots & \overline{f_{1k}}\\
\vdots & \vdots & \vdots \\
\overline{f_{k1}} & \ldots & \overline{f_{kk}}
\end{array}\right).$$
\end{proof}

\section{Dualities}

Let $G$ be an algebraic supergroup and $\sigma$ be an anti-automorphism of the Hopf superalgebra $K[G]$.
If $M$ is a finite-dimensional $G$-supermodule, then one can define its $\sigma$-{\it dual} $M^{<\sigma>}$ as follows (cf.  \cite{zub1}).
Fix a homogeneous basis of $M$ consisting of elements $m_i$ for $1\leq i\leq t$.  If $\tau_M(m_i)=\sum_{1\leq k\leq t} m_k\otimes f_{ki}$, where 
$f_{ki}\in K[G]$ and $|f_{ki}|=|m_i|+|m_k| \pmod 2$, then the $G$-supermodule $M^{<\sigma>}$ has a basis consisting of elements $m^{<\sigma>}_i$ such that
$$\tau_{M^{<\sigma>}}(m^{<\sigma>}_i)=\sum_{1\leq k\leq t} m^{<\sigma>}_k \otimes (-1)^{|m_k|(|m_i|+|m_k|)}\sigma(f_{ik}).$$ 
The functor $M\to M^{<\sigma>}$ is a self-duality of the full subcategory of all finite-dimensional supermodules.

Moreover, $\sigma$ induces an anti-automorphism of $Dist(G)$ by $\phi\mapsto\phi\cdot\sigma=\phi^{<\sigma>}$ for $\phi\in Dits(G)$. In other words,
$(\phi\psi)^{<\sigma>}=(-1)^{|\phi||\psi|}\psi^{<\sigma>}\phi^{<\sigma>}$ for every $\phi, \psi\in Dist(G)$.
\begin{lm}\label{actofdistontaudual}
$M^{<\sigma>}$ can be identified with the dual superspace $M^*$ on which $Dist(G)$ acts by
$(\phi f)(v)=(-1)^{|\phi||f|}f(\phi^{<\sigma>} m)$ for $f\in M^*$ and $m\in M$.
\end{lm}
\begin{proof}
Let $m^*_i$ form a dual basis of $M^*$. The required isomorphism $M^*\to M^{<\sigma>}$ is defined by $m^*_i\mapsto 
m^{<\sigma>}_i$. 

\end{proof} 
\begin{rem}
If $\sigma=s_{K[G]}$, then $M^{<\sigma_G>}$ coincides with the dual $G$-supermodule $M^*$.
In particular, $Dist(G)$ acts on $M^*$ by $(\phi f)(m)=(-1)^{|\phi||f|}f(s_{Dist(G)}(\phi)m)$.
\end{rem}
Assume $G=GL(m|n)$.
Then the map $t : c_{ij}\mapsto (-1)^{|i|(|i|+|j|)}c_{ji}$ induces an anti-automorphism of the Hopf superalgebra $K[G]$.
Furthemore, this anti-automorphism induces a self-duality $M\mapsto M^{<t>}$ of the full subcategory of all finite-dimensional $G$-supermodules.  
For example, $H^0(\lambda)^{<t>}=V(\lambda)$ (see \cite{zub1} for more details).

The related anti-automorphism $\phi\mapsto\phi^{<t>}$ of the superalgebra $Dist(G)$ is defined by $(e_{ij}^{(l)})^{<t>} =e_{ji}^{(l)}$ if $e_{ij}$ is even, 
and $e_{ij}^{<t>} =(-1)^{|j|(|i|+|j|)}e_{ji}$ otherwise.

\section{Good and Weyl filtrations}

Let $N$ be a normal subsupergroup of an algebraic supergroup $G$ and  
$M$ be a finite-dimensional $G$-supermodule. Denote $M_{Dist(N)}$ by $M_N$. 
\begin{lm}\label{specificsubmod}
If $G$ is connected, then $M_N$ is a $G$-subsupermodule. 
\end{lm}
\begin{proof}
Combine Lemma \ref{normalhopf} and Lemma 9.4 of \cite{zub2}. 
\end{proof}
\begin{cor}\label{corroftheabovelemma}
The statement of Lemma \ref{specificsubmod} holds for every $G$-supermodule $M$.
\end{cor}
\begin{proof}
The supermodule $M$ coincides with a union of its finite-dimensional subsupermodules $M_{\alpha}$, where $\alpha$ runs over a directed set $I$ such that $M_{\alpha}\subseteq M_{\beta}$ if and only if $\alpha\leq\beta$. We have 
$M_N=\bigcup_{\alpha\in I} (M_{\alpha})_N$. By Lemma \ref{specificsubmod}, each $(M_{\alpha})_N$ is a $G$-subsupermodule. Thus $M_N$ is a $G$-subsupermodule.
\end{proof}
\begin{rem}\label{asubmodule}
$M_N$ is the smallest subsupermodule of $G$ such that $N$ acts identically on $M/M_N$.
\end{rem}
\begin{rem}\label{inflation/restriction}
Let $\pi : G\to H$ be an epimorphism of algebraic supergroups. Then
every $H$-supermodule $M$ can be regarded as a $G$-supermodule via $\pi$. In \cite{don3} this $G$-supermodule is denoted by $\pi_0(M)$. Additionally, if $G=N\rtimes H$, then 
$\pi_0(M)|_{H}\simeq M$. 
\end{rem}

From now on assume that $G=GL(m|n)$ and $P=Stab_G(V_1)$. Let $U$ denote the kernel of $P\to G_{ev}$. Then
$P=U\rtimes G_{ev}$ (cf. Remark 5.2 of \cite{zub1}).
Symmetrically, denote $Stab_G(V_0)$ by $P^{opp}$. As above, we have an epimorphism $P^{opp}\to G_{ev}$ and its kernel is denoted by 
$U^{opp}$. Besides, $P^{opp}=U^{opp}\rtimes G_{ev}$. Both supergroups $U$ and $U^{opp}$ are obviously infinitesimal and unipotent.
 
A standard supermodule $V(\lambda)$ can be also defined as a universal supermodule of the highest weight $\lambda$. In other words, 
$V(\lambda)$ is generated by a $B^{opp}$-{\it primitive vector} of weight $\lambda$ and if a $G$-supermodule $M$ is generated by a $B^{opp}$-primitive vector of weight $\lambda$, then there is an epimorphism $V(\lambda)\to M$.

By Remark \ref{inflation/restriction}, a $G_{ev}$-(super)module $N$ can be regarded as a $P^{opp}$-supermodule as well. By Corollary 3.5 of \cite{brunkuj}, $Dist(G)\otimes_{Dist(P^{opp})} N$ is a $G$-supermodule.
\begin{lm}\label{Weylmod}
For every dominant weight $\lambda$, the Weyl supermodule $V(\lambda)$ is isomorphic to $Dist(G)\otimes_{Dist(P^{opp})} V_{ev}(\lambda)$.
\end{lm}
\begin{proof}
Denote a $G$-supermodule $Dist(G)\otimes_{Dist(P^{opp})} V_{ev}(\lambda)$ by $M$.
The formal character of $M$ coincides with the formal character of $V(\lambda)$ 
(cf. Proposition 5.9 and Theorem 5.4 of \cite{zub1}). 
Since $M$ is generated by a primitive vector of weight $\lambda$, 
it is an epimorphic image of $V(\lambda)$. Therefore $M\simeq V(\lambda)$. 
\end{proof}
\begin{pr}\label{sequence}
For each $G$-supermodule $M$ and each $G_{ev}$-supemodule $N$ there is the following spectral sequence
$$E^{n,m}_2=Ext^n_{G_{ev}}(N, H^m(U^{opp}, M))\Rightarrow Ext^{n+m}_{G}(Dist(G)\otimes_{Dist(P^{opp})} N, M).$$
\end{pr}
\begin{proof}
Using (co)Frobenius reciprocity law we have
$$Hom_G(Dist(G)\otimes_{Dist(P^{opp})} N, M)\simeq Hom_{P^{opp}}(N, M|_{P^{opp}}).$$
Let $I^{\bullet}$ be an injective resolution of $G$-supermodule $M$.
Since $P^{opp}_{ev}=G_{ev}$, $G/P^{opp}$ is an affine superscheme, hence $P^{opp}$ is exact (even faithfully exact) in $G$ (see also Proposition 5.1 of \cite{zub1}). Thus $I^{\bullet}|_{P^{opp}}$ is an injective resolution of $P_{opp}$-supermodule
$M|_{P^{opp}}$ such that the complexes $Hom_G(Dist(G)\otimes_{Dist(P^{opp})} N, I^{\bullet})$ and $Hom_{P^{opp}}(N, I^{\bullet}|_{P^{opp}})$ are isomorphic to each other. Then
$$Ext^i_G(Dist(G)\otimes_{Dist(P^{opp})} N, M)\simeq Ext^i_{P^{opp}}(N, M|_{P^{opp}}).$$
The statement now follows from Proposition 3.1 (2) of \cite{zubscal}.
\end{proof}
\begin{rem}\label{aconsequence}
If $M|_{U^{opp}}$ is an injective $U^{opp}$-supermodule, then the above spectral sequence degenerates and there is 
an isomorphism  
$$Ext^n_{G_{ev}}(N, M^{U^{opp}})\simeq Ext^{n}_{G}(Dist(G)\otimes_{Dist(P^{opp})} N, M)$$
for all $n\geq 0$.
\end{rem}
By Lemma 5.1 of \cite{zub1}, for every $P$-supermodule $M$ there is an isomorphism of $U^{opp}$-supermodules $M\otimes K[U^{opp}]\to ind^G_P M$. Moreover, 
it is also an isomorphism of $G_{ev}$-(super)modules, where $G_{ev}$ acts on
$K[U^{opp}]$ by conjugations. In particular, $ind^G_P M$ is an injective $U^{opp}$-supermodule and the $G_{ev}$-supermodule
$(ind^G_P M)^{U^{opp}}$ is isomorphic to $M|_{G_{ev}}$.

For example, Lemma 5.2 from \cite{zub1} states $H^0(\lambda)\simeq ind^G_P H^0_{ev}(\lambda)$.  Thus $H^0(\lambda)$ is an injective $U^{opp}$-supermodule and $H^0(\lambda)^{U^{opp}}\simeq H^0_{ev}(\lambda)$. These elementary observations inspire the following theorem.

Let $\Gamma$ be a finitely generated ideal of (dominant) weights and $M$ be a $\Gamma$-restricted $G$-supermodule (cf. \cite{zubmar}). 
\begin{tr}\label{goodfiltr}
$M$ has a decreasing good filtration if and only if
$G_{ev}$-supermodule (or $G_{res}$-module) $M^{U^{opp}}$ has a decreasing good filtration and $H^1(U^{opp}, M)=0$.
\end{tr}
\begin{proof}
Apply Proposition \ref{sequence} to $N=V_{ev}(\lambda)$ and obtain the five-term exact sequence (cf. p.50 of \cite{jan}):
$$0\to Ext^1_{G_{ev}}(V_{ev}(\lambda), M^{U^{opp}})\to Ext^1_{G}(V(\lambda), M)\to
Hom_{G_{ev}}(V_{ev}(\lambda), H^1(U^{opp}, M))\to
$$
$$
\to Ext^2_{G_{ev}}(V_{ev}(\lambda), M^{U^{opp}})\to Ext^2_{G}(V(\lambda), M).
$$
By Theorem 4.9 and Lemma 4.19 of \cite{zubmar}, $M$ has a decreasing good filtration if and only if $Ext^1_{G}(V(\lambda), M)=0$ for every (dominant) weight
$\lambda$. Furthemore, in this case $Ext^i_{G}(V(\lambda), M)=0$ for every $i\geq 1$.
Thus  $Ext^1_{G_{ev}}(V_{ev}(\lambda), M^{U^{opp}})=0$, and $M^{U^{opp}}$ has a decreasing good filtration as a $G_{ev}$-supermodule, provided $M$ has a decreasing good filtration
as a $G$-supermodule. 
The latter condition implies $Ext^2_{G_{ev}}(V_{ev}(\lambda), M^{U^{opp}})=0$ and  
$Hom_{G_{ev}}(V_{ev}(\lambda), H^1(U^{opp}, M))=0$ for every 
$\lambda$. On the other hand, if $H^1(U^{opp}, M)\neq 0$, then there is a weight $\lambda$ such that $L_{ev}(\lambda)$ is a direct summand of $soc(H^1(U^{opp}, M))$. Thus
$Hom_{G_{ev}}(V_{ev}(\lambda), H^1(U^{opp}, M))\neq 0$, which is a contradiction. 
Conversely, if $M$ satisfies the conditions of this theorem, then the first and the third member of the above exact sequence are zeroes, hence $Ext^1_G(M, V(\lambda))=0$.
\end{proof}
\begin{rem}\label{asinjective}
By Lemmas \ref{injandproj} and \ref{injandproj2}, $M|_{U^{opp}}$ is injective and isomorphic to  $K[U^{opp}]^{\oplus I}\bigoplus \Pi K[U^{opp}]^{\oplus J}$ for (possibly infinite) index sets $I$ and $J$. If $M$ is finite-dimensional, then the cardinalities $I^{\sharp}$ and $J^{\sharp}$ of $I$ and $J$ are given as
$$I^{\sharp}=\sum_{\lambda, (M : H^0(\lambda^0))\neq 0}\dim H^0_{ev}(\lambda) \mbox{ and } J^{\sharp}=\sum_{\lambda, (M : H^0(\lambda^1))\neq 0}\dim H^0_{ev}(\lambda).$$ 
In fact, if $M$ has a good filtration
$$0\subseteq M_1\subseteq\ldots\subseteq M_s=M$$
such that $M_i/M_{i-1}\simeq H^0(\lambda_i^{a_i})$, where $a_i=0, 1$ and $1\leq i\leq s$, then $M^{U^{opp}}$ has a good filtration
$$0\subseteq M_1^{U^{opp}}\subseteq\ldots\subseteq M_s^{U^{opp}}=M^{U^{opp}}$$
such that $M_i^{U^{opp}}/M_{i-1}^{U^{opp}}\simeq H_{ev}^0(\lambda^{a_i}_i)$, where $1\leq i\leq s$. Besides, $M\simeq\bigoplus_{1\leq i\leq s} M_i/M_{i-1}$ as $U^{opp}$-supermodule.
\end{rem}

Observe that the anti-automorphism $\phi\to \phi^{<t>}$ of $Dist(G)$  induces an anti-isomorphism between subsuperalgebras $Dist(U)$ and $Dist(U^{opp})$. 
Moreover, if $M$ is a $Dist(U)$-supermodule, then
the dual superspace $M^*$ has a natural structure of a $Dist(U^{opp})$-supermodule given by 
$$(\phi f)(m)=(-1)^{|\phi||f|}(\phi^{<t>} m), \mbox{ where } \phi\in Dist(U^{opp}), f\in V^* \mbox{ and } m\in M.$$
Symmetrically, if $M$ is a $Dist(U^{opp})$-supermodule, then $M^*$ has a structure of $Dist(U)$-supermodule via the same rule.
We denote $M^*$ by $M^{<t>}$, no matter over which superalgebra the supermodule $M$ is defined. The following lemma is now obvious.
\begin{lm}\label{self-dualitybetween}
The functor $M\mapsto M^{<t>}$ is an anti-equivalence between the categories of finite-dimensional $U$-supermodules and $U^{opp}$-supermodules.  
\end{lm}
\begin{proof}
Since $Dist(U)=K[U]^*$ and $Dist(U^{opp})=K[U^{opp}]^*$, the result follows from Lemma \ref{equivalence}.
\end{proof}
Observe that if $M$ is a $G$-supermodule, then $M^{<t>}|_{U} = (M|_{U^{opp}})^{<t>}$, and symmetrically, 
$M^{<t>}|_{U^{opp}} = (M|_{U})^{<t>}$.
\begin{pr}\label{duality}
Let $M$ be a finite-dimensional $U$-supermodule. Then for every $k\geq 0$ there is a natural isomorphism 
of superspaces $H_k(U, M)^{<t>}\simeq H^k(U^{opp}, M^{<t>})$. Moreover, if $M$ is a $G$-supermodule, then 
it is an isomorphism of $G_{ev}$-supermodules.
\end{pr}
\begin{proof}
An element $f$ belongs to $(M^{<t>})^{U^{opp}}$ if and only if for every $\phi\in Dist(U^{opp})^+$ one has $\phi f=0$; meaning that
for every $m\in M$ one has $f(\phi^{<t>}m)=0$. The last condition is equivalent to $f\in (M/M_U)^{<t>}$. 
Since the functor $V\mapsto V^{<t>}$ maps projective resolutions to injective resolutions, both statements follow.  
\end{proof}
\begin{tr}\label{Weylfiltr}
A finite-dimensional $G$-supermodule $M$ has a Weyl filtration if and only if $G_{ev}$-supermodule (or $G_{res}$-module) $M/M_U$ has a Weyl filtration and
$H_1(U, M)=0$.
\end{tr}
\begin{proof}
Apply Theorem \ref{goodfiltr} to the $G$-supermodule $M^{<t>}$ and refer to Proposition \ref{duality}.
\end{proof}
\begin{rem}\label{proj}
By Lemmas \ref{injandproj} and \ref{injandproj2}, $M|_U$ is projective and isomorphic to $Dist(U)^{\oplus I}\bigoplus \Pi Dist(U)^{\oplus J}$. As in Remark \ref{asinjective}, if $M$ has a Weyl filtration
$$0\subseteq M_1\subseteq\ldots\subseteq M_s=M$$
such that $M_i/M_{i-1}\simeq V(\lambda_i^{a_i})$ for $a_i=0, 1$ and $1\leq i\leq s$, then $M/M_U$ has a Weyl filtration
$$0\subseteq M_1/(M_1)_U\subseteq\ldots\subseteq M_s/(M_s)_U=M/M_U$$
such that $M_i/(M_i)_U/M_{i-1}/(M_{i-1})_U\simeq (M_i/M_{i-1})/(M_i/M_{i-1})_U=V_{ev}(\lambda^{a_i}_i)$ for $1\leq i\leq s$. 
Additionally, $M\simeq\bigoplus_{1\leq i\leq s} M_i/M_{i-1}$ as an $U$-supermodule.
Therefore
$$I^{\sharp}=\sum_{\lambda, (M : V(\lambda^0))\neq 0}\dim V_{ev}(\lambda) \mbox{ and } J^{\sharp}=\sum_{\lambda, (M : V(\lambda^1))\neq 0}\dim V_{ev}(\lambda).$$ 
\end{rem}

Corresponding to a decreasing chain of finitely generated ideals 
$$\Gamma=\Gamma_0\supseteq\Gamma_1\supseteq\Gamma_2\supseteq\ldots$$
such that $\Gamma\setminus\Gamma_k$ is finite for every $k\geq 0$ and $\bigcap_{k\geq 0}\Gamma_k=\emptyset$, there is an increasing chain of finite-dimensional subsupermodules
$$0\subseteq O^{\Gamma_1}(M)\subseteq O^{\Gamma_2}(M)\subseteq\ldots .$$

By Theorem 4.11 of \cite{zubmar}, $M$ has an increasing Weyl filtration if and only if each $M_k=O^{\Gamma_k}(M)$ has an increasing Weyl filtration if and only if each $M_k^{<t>}$ has a good filtration. 
\begin{cor}\label{Weylfiltrcor}
$M$ has an increasing Weyl filtration if and only if for every $k\geq 1$ $H_1(U, M_k)=0$ and $M_k/(M_k)_U$ has a Weyl filtration as a $G_{ev}$-supermodule (or as a $G_{res}$-module).
\end{cor}
Analogously as in Remark \ref{proj}, we derive that $M\simeq \bigoplus_{k\geq 1} M_k/M_{k-1}$ is a  projective $U$-supermodule. Moreover, $M/M_U$ has an increasing  Weyl filtration consisting of submodules $M_k/(M_k)_U$.
\begin{question}\label{converse}
Assume that $M$ is projective as $U$-supermodule and $G_{ev}$-supermodule $M/M_U$ has an increasing Weyl filtration. Does it imply
that $M$ has an increasing Weyl filtration as a $G$-supermodule?
\end{question}
We have proved the following theorem.
\begin{tr}\label{criterion}
A restricted $G$-supermodule $M$ is tilting if and only if the following two conditions are satisfied.
\begin{enumerate}
\item The $G_{ev}$-supermodule $M^{U^{opp}}$ 
has a decreasing good filtration and $M|_{U^{opp}}$ is injective;
\item For any decreasing chain of finitely generated ideals 
$$\Gamma=\Gamma_0\supseteq\Gamma_1\supseteq\Gamma_2\supseteq\ldots$$
such that $\Gamma\setminus\Gamma_k$ is finite for every $k\geq 0$ and $\bigcap_{k\geq 0}\Gamma_k=\emptyset$, 
each $M_k=O^{\Gamma_k}(M)$ is projective as $U$-supermodule and $M_k/(M_k)_U$ has a Weyl filtration as a $G_{ev}$-supermodule.
\end{enumerate}
\end{tr}

\section{Representations of $G_r T$ and $G_r B$}

Following \cite{jan}, we denote $ind^{G_r B}_B K^{a}_{\lambda}$ and $Dist(G_r B^{opp})\otimes_{Dist(B^{opp})} K^{a}_{\lambda}$
by $\Hat{Z}'_r (\lambda^a)$ and $\Hat{Z}_r (\lambda^a)$ respectively. Without a loss of generality, one can assume that $a=0$. 

By Lemma 14.2 of \cite{zub4} there are superscheme isomorphisms 
$$B\times V^{opp}_r\simeq BG_r=G_r B \mbox{ and } B^{opp}\times V_r\simeq B^{opp}G_r =G_r B^{opp}.$$ 
Thus $Dist(G_r B^{opp})$ is a free $Dist(B^{opp})$-supermodule (see Lemma 5.2 of \cite{kuj}). 
Combining this and Theorem 10.1 of \cite{zub4}, we obtain that $\Hat{Z}'_r (\lambda)|_{G_r}\simeq Z'_r (\lambda)$ and $\Hat{Z}_r (\lambda)|_{G_r}\simeq Z_r (\lambda)$. 
Analogously, there are natural superscheme isomorphims
$$B_r T\times V^{opp}_r\simeq G_r T \mbox{ and } B^{opp}_r T\times V_r\simeq G_r T$$ 
such that $\Hat{Z}'_r (\lambda)|_{G_r T}\simeq ind^{G_r T}_{B_r T} K_{\lambda}$ and $\Hat{Z}_r (\lambda)|_{G_r T}\simeq
Dist(G_r T)\otimes_{Dist(B^{opp}_r T)} K_{\lambda}$. 

As in the fifth section, we see that $\Hat{Z}'_r (\lambda)|_{B^{opp}_r T}\simeq K_{\lambda}\otimes K[V^{opp}_r]$ and
$\Hat{Z}'_r (\lambda)^{V^{opp}_r}=\Hat{Z}'_r (\lambda)_{\lambda}\simeq K_{\lambda}$. Thus the socle of $\Hat{Z}'_r (\lambda)$ is irreducible and generated by
a $B^{opp}_r T$-primitive vector of weight $\lambda$. Observe also that $\Hat{Z}'_r (\lambda)_{\mu}\neq 0$ implies $\mu\leq\lambda$. 

Analogously, $\Hat{Z}_r (\lambda)|_{B_r T}\simeq Dist(V_r)\otimes K_{\lambda}$. Since the $B_r T$-top of 
$\Hat{Z}_r (\lambda)$ is isomorphic to  
$\Hat{Z}_r (\lambda)/\Hat{Z}_r (\lambda)_{V_r}\simeq K_{\lambda}$, the $G_r B^{opp}$-top of $\Hat{Z}_r (\lambda)$ is also irreducible. 
Moreover, $\Hat{Z}_r (\lambda)$ is generated by a $B^{opp}$-primitive vector of weight $\lambda$ and
$\Hat{Z}_r (\lambda)_{\mu}\neq 0$ implies $\mu\leq\lambda$.

Denote by 
$\Hat{L}'_r(\lambda)$ and $\Hat{L}_r(\lambda)$ the modules $soc_{G_r B}\Hat{Z}'_r (\lambda)$ and 
$\Hat{Z}_r (\lambda)/rad_{G_r B^{opp}}\Hat{Z}_r (\lambda)$, respectively.
\begin{lm}\label{universalandcouniversal}
The supermodules $\Hat{Z}'_r (\lambda)$ (and $\Hat{Z}_r (\lambda)$, respectively) are couniversal (and universal, respectively) objects in the category
$G_r B -\mathsf{SMod}$ (and $G_r B^{opp} -\mathsf{SMod}$, respectively). In particular, each irreducible $G_r B$-supermodule
is isomorphic to exactly one 
$\Hat{L}'_r(\lambda^a)$ and each irreducible $G_r B^{opp}$-supermodule
is isomorphic to exactly one 
$\Hat{L}_r(\lambda^a)$. 
\end{lm}
\begin{proof}
The statement related to $\Hat{Z}_r (\lambda)$ is obvious.
Let $M$ be a $G_r B$-supermodule that is cogenerated by a $B^{opp}_r T$-primitive vector $v$ of weight $\lambda$ and $M_{\mu}\neq 0$ implies $\mu\leq\lambda$. 
Consider a vector $m\in M^{V^{opp}_r}\setminus 0$ of weight $\mu$. Set $M'=Dist(G_r B) m=Dist(V)m$ (this is a $G_r B$-subsupermodule generated by $m$; see Lemma 9.4 of \cite{zub2}).  
Then $M'_{\mu}=Km$ and $M'_{\pi}\neq 0$ implies $\pi\leq\mu$. Since $v\in M'$, it follows that $\mu=\lambda$ and
$M^{V^{opp}_r}=M_{\lambda}=Kv$. In particular, $soc_{G_r B} M=\Hat{L}_r(\lambda^{|v|})$.

Furthermore, $M|_B$ has a factor that is isomorphic to 
$K_{\lambda}$ (cf. Remark 5.3 of \cite{zub1}), hence
$Hom_{G_r B}(M, \Hat{Z}'_r (\lambda))\simeq Hom_{B}(M, K_{\lambda})=K$ and  $M$ is embedded into $\Hat{Z}'_r (\lambda)$.
\end{proof}

We leave for the reader to verify that 
the socle of $\Hat{Z}'_r (\lambda)|_{G_r T}$ and the top of
$\Hat{Z}_r (\lambda)|_{G_r T}$ are irreducible $G_r T$-supermodules. Furthemore, $\Hat{Z}'_r (\lambda)|_{G_r T}$ and 
$\Hat{Z}_r (\lambda)|_{G_r T}$, respectively are couniversal and universal objects, respectively, in $G_r T-\mathsf{SMod}$. Therefore every irreducible $G_r T$-supermodule is isomorphic to the socle of exactly one $\Hat{Z}'_r (\lambda^a)|_{G_r T}$ 
and to the top of exactly one $\Hat{Z}_r (\lambda^a)|_{G_r T}$, respectively.

The anti-automorphism $\phi\mapsto\phi^{<t>}$ maps $Dist(G_r B)$ onto $Dist(G_r B^{opp})$. In particular, the functor $M\mapsto M^{<t>}$ induces a duality between the full subcategory of $G_r B -\mathsf{SMod}$, consisting of all finite-dimensional supermodules, and the same kind subcategory in $G_r B^{opp} -\mathsf{SMod}$. Since
$M^{<t>}|_T\simeq M|_T$ for every $T$-supermodule $M$ (cf. Lemma 5.4 of \cite{zub1}), this implies $\Hat{Z}'_r (\lambda^a)^{<t>}\simeq \Hat{Z}_r (\lambda^a)$ and
$\Hat{L}'_r (\lambda^a)^{<t>}\simeq \Hat{L}_r (\lambda^a)$. Moreover, $\phi\mapsto\phi^{<t>}$ induces an anti-automorphism of
$Dist(G_r T)$, hence it induces a self-duality of the full subcategory of $G_r T-\mathsf{SMod}$, consisting of all finite-dimensional supermodules,  such that $(\Hat{Z}'_r (\lambda^a)|_{G_r T})^{<t>}\simeq \Hat{Z}_r (\lambda^a)|_{G_r T}$. 

The following lemma generalizes Proposition II.9.6 from \cite{jan}.
\begin{lm}\label{simplesoverG_rT}
The socle of $\Hat{Z}'_r (\lambda)|_{G_r T}$ is isomorphic to the top of $\Hat{Z}_r (\lambda)|_{G_r T}$.
If $\Tilde{L}_r(\lambda)$ denotes this irreducible $G_r T$-supermodule, then $\Tilde{L}_r(\lambda)^{<t>}\simeq \Tilde{L}_r(\lambda)$.

Also, 
$$soc_{G_r}\Hat{Z}'_r (\lambda)=soc_{G_r T} \Hat{Z}'_r (\lambda)=soc_{G_r B} \Hat{Z}'_r (\lambda)$$ and
$$rad_{G_r} \Hat{Z}_r (\lambda)=rad_{G_r T} \Hat{Z}_r (\lambda)=rad_{G_r B^{opp}} \Hat{Z}_r (\lambda).$$
\end{lm}
\begin{proof}
Since $(soc_{G_r T} \Hat{Z}'_r (\lambda))^{V_r^{opp}}\simeq K_{\lambda}$, the first statement follows by the same arguments
as in the fifth section. 

Let $v$ denotes a generator of $\Hat{Z}'_r (\lambda)_{\lambda}$.
Then $$soc_{G_r T} \Hat{Z}'_r (\lambda)=
Dist(V_r) v=soc_{G_r T}\Hat{Z}'_r (\lambda)\subseteq Dist(V)v=soc_{B_r T}\Hat{Z}'_r (\lambda).$$ 
On the other hand, $Dist(G_r B)=Dist(B_{ev})Dist(G_r)$. Thus
$soc_{G_r B} \Hat{Z}'_r (\lambda)$ is completely reducible as $G_r$-supermodule, and therefore it is contained in $soc_{G_r T}\Hat{Z}'_r (\lambda)$.
Application of the duality $M\mapsto M^{<t>}$ implies the last statement.
\end{proof}

\section{Steinberg supermodules}

The positive even roots of $G$ are $\epsilon_i-\epsilon_j$, where
$1\leq i < j\leq m$ or $m+1\leq i < j\leq m+n$ and the positive odd roots of $G$ are $\epsilon_i-\epsilon_j$ for $1\leq i\leq m < j\leq m+n$. 
Let $\rho_0$ denote the half of the sum of all positive even roots of $G$, 
$\rho_1$ denote the half of the sum of all positive odd roots of $G$, $\rho=\rho_0-\rho_1$, and $\rho_{s, t}$ denotes a weight $s\sum_{1\leq i\leq m}\epsilon_i+
t\sum_{m+1\leq i\leq m+n}\epsilon_i$ for $s, t\in\mathbb{Z}$. 
Then $\rho_1=\frac{1}{2}\rho_{n, -m}$. 
\begin{lm}\label{a translation}
For every dominant weight $\lambda$ and for every positive integer $r$ there is an isomorphism
$$H^0((p^r-1)\rho_0 +\rho_{s, t})\otimes H^0_{ev}(\lambda)^{[r]}\simeq H^0((p^r-1)\rho_0 +\rho_{s, t}+p^r\lambda).$$
\end{lm}
\begin{proof}
The module $H^0_{ev}(\lambda)^{[r]}$, considered as a $G$-supermodule, is an even Frobenius twist of $H^0_{ev}(\lambda)$.
The tensor identity implies
$$H^0((p^r-1)\rho_0 +\rho_{s, t})\otimes H^0_{ev}(\lambda)^{[r]}\simeq ind^G_{P} (H^0_{ev}((p^r-1)\rho_0 +\rho_{s, t})\otimes H^0_{ev}(\lambda)^{[r]}).$$
The simple $G_{ev}$-module $L_{ev}(\rho_{s, t})$ is one-dimensional. In fact, it is isomorphic to $\det(C_{00})^s\otimes
\det(C_{11})^t$. Using Proposition II.3.19 of \cite{jan} and the tensor identity we obtain
$$H^0_{ev}((p^r-1)\rho_0 +\rho_{s, t})\otimes H^0_{ev}(\lambda)^{[r]}\simeq H^0_{ev}((p^r-1)\rho_0)\otimes H^0_{ev}(\lambda)^{[r]}\otimes L_{ev}(\rho_{s, t})\simeq
$$ 
$$
H^0_{ev}((p^r-1)\rho_0 +p^r\lambda)\otimes L_{ev}(\rho_{s, t})\simeq H^0_{ev}((p^r-1)\rho_0+p^r\lambda +\rho_{s, t}).
$$
\end{proof}
\begin{rem}\label{forWeyl}
Application of the functor $M\to M^{<t>}$ gives
$$V((p^r-1)\rho_0 +\rho_{s, t})\otimes V_{ev}(\lambda)^{[r]}\simeq V((p^r-1)\rho_0 +\rho_{s, t}+p^r\lambda).$$
\end{rem}
\begin{lm}\label{irreducible}
The supermodule $H^0((p^r-1)\rho_0+\rho_{s, t})$ is irreducible if and only if 
$p\not| (\frac{(p^r +1)}{2}m+\frac{(p^r -1)}{2}n+s+t)$.
\end{lm}
\begin{proof}
Combining Theorem 1 of \cite{marko} and Remark II.3.19 from \cite{jan} we see that $H^0((p^r-1)\rho_0)+\rho_{s, t})$ is irreducible if and only if 
$p\not| ((p^r-1)\rho_0)+\rho_{s, t}+\rho, \alpha)$ for every positive odd root $\alpha$.
If $\alpha=\epsilon_i-\epsilon_j$ and $1\leq i\leq m < j\leq m+n$, then 
$$((p^r-1)\rho_0)+\rho_{s, t}+\rho, \alpha)=
\frac{(p^r -1)}{2}(m+n)+s + t+ m +p^r(m-i-j-1)
$$
and the lemma follows.
\end{proof}
An irreducible $G$-supermodule from Lemma \ref{irreducible} is called an $r$-th {\it Steinberg supermodule}.
It is obvious that an $r$-th Steinberg supermodule remains irreducible as a $G_r$-supermodule.

Let $T(\lambda)$ denotes the indecomposable tilting $G$-supermodule of the highest weight $\lambda$ (cf. \cite{zubmar}). Let $T_{ev}(\lambda)$ denotes the indecomposable 
tilting $G_{ev}$-supermodule (or $G_{res}$-module) of the highest weight $\lambda$. 
The following proposition generalizes Proposition (2.1) from \cite{don2} to general linear supergroups.
\begin{pr}\label{asdonkin}
If $H^0((p^r-1)\rho_0+\rho_{s, t})$ is a Steinberg supermodule, then 
$$H^0((p^r-1)\rho_0+\rho_{s, t})\otimes T_{ev}(\lambda)^{[r]}\simeq T((p^r-1)\rho_0+\rho_{s, t}+p^r\lambda).$$
\end{pr}
\begin{proof}
Denote $(p^r-1)\rho_0+\rho_{s, t}$ by $\pi$.
Since $H^0(\pi)=L(\pi)=V(\pi)$,
Lemma \ref{a translation} and Remark \ref{forWeyl} imply that $T=H^0(\pi)\otimes T_{ev}(\lambda)^{[r]}$ has both good and Weyl filtrations. 
Moreover, $\pi+p^r\lambda$ is the unique highest weight of $T$. It remains to prove that $T$ is indecomposable (cf. Theorem 4.17 and Theorem 4.19 of \cite{zubmar}; see also Theorem 1.1 of \cite{don2}). 

We have $End_G(T)\simeq (T^*\otimes T)^G =(T^*\otimes T)^{Dist(G)}$. The action of $Dist(G)$ on $End_K(T)$ 
can be described as $(\psi\cdot\phi)(v)=(-1)^{|\phi||\psi|}\psi \phi(s_{Dist(G)}(\psi) v))$ for $\psi\in Dist(G)$ and $\phi\in End_K(T)$.
We will follow the idea from Lemma in \S 2 of \cite{don1}. Since $G_r$ acts on $T_{ev}(\lambda)^{[r]}$ trivially, there is an isomorphism
of superalgebras
$$End_{G_r} (T)\simeq End_{G_r}(L_r((\pi))\otimes End_K(T_{ev}(\lambda)^{[r]}).$$
Additionally, we have $L_r(\pi)^{V^{opp}_r}= Z'_r(\pi)^{V^{opp}_r}=K_{\mu}$, hence $End_{G_r}(L_r(\pi))=K id_{L(\pi)}$. 

Finally, 
$$End_G(T)=End_{G_r}(T)^G\simeq K id_{L(\pi)}\otimes End_G(T_{ev}(\lambda)^{[r]})\simeq End_G(T_{ev}(\lambda)^{[r]}).$$
Using Lemma \ref{endomorphismring} we derive that $End_G(T_{ev}(\lambda)^{[r]})\simeq End_{G_{res}}(T_{ev}(\lambda))$ is a local purely even superalgebra, 
hence $T$ is indecomposable.
\end{proof}
\begin{lm}\label{restrictionon}
Let $H^0((p^r-1)\rho_0+\rho_{s, t})$ be a Steinberg supermodule. Then
$$H^0((p^r-1)\rho_0+\rho_{s, t})|_{G_r B}\simeq \Hat{L}'_r((p^r-1)\rho_0+\rho_{s, t})$$ and
$$H^0((p^r-1)\rho_0+\rho_{s, t})|_{G_r B^{opp}}\simeq \Hat{L}_r((p^r-1)\rho_0+\rho_{s, t}).$$
Thus 
$$H^0((p^r-1)\rho_0+\rho_{s, t})|_{G_r T}\simeq \Tilde{L}_r((p^r-1)\rho_0+\rho_{s, t}).$$
\end{lm}
\begin{proof}
Denote $(p^r-1)\rho_0+\rho_{s, t}$ by $\pi$, again. Lemma \ref{universalandcouniversal} implies that $H^0(\pi)|_{G_r B}$ is embedded into $\Hat{Z}'_r(\pi)$. 
On the other hand, $\dim H^0(\pi)=\dim K[U^{opp}]\dim H^0_{ev}(\pi)=\dim K[V^{opp}_r]=\dim \Hat{Z}'_r(\pi)$ (cf. II.9.16 of \cite{jan}). 
Thus $H^0(\pi)|_{G_r B}\simeq \Hat{Z}'_r(\pi)$ and by Lemma \ref{simplesoverG_rT} we obtain
$\Hat{Z}'_r(\pi)=\Hat{L}'_r(\pi)$.

Since $H^0(\pi)^{<t>}=
L(\pi)^{<t>}\simeq L(\pi)=H^0(\pi)$, the second statement follows.
The last statement follows by Lemma \ref{simplesoverG_rT}.
\end{proof}
The proof of the following lemma can be modified from the proof of Proposition II.2.14 in \cite{jan}.
\begin{lm}\label{asinjan}
If $\lambda\not<\mu$, then 
$$Ext^1_{G_r T}(\Tilde{L}_r(\lambda), \Tilde{L}_r(\mu))\simeq Hom_{G_r T}(rad_{G_r T}\Hat{Z}_r(\lambda), \Tilde{L}_r(\mu)).$$
\end{lm}
For any finite-dimensional $G_r T$-supermodules $M_1$ and $M_2$ there is a (superspace) isomorphisms
$$Ext^{n}_{G_r T}(M_1, M_2)\simeq Ext^{n}_{G_r T}(M_2^{<t>}, M_1^{<t>}), n\geq 0.$$
We will only sketch a proof of this statement. By Proposition 3.2 from \cite{zub1} (see also Lemma I.4.4 of \cite{jan}), there is a superspace isomorphism
$$Ext^n_{G_r T}(M_1, M_2)\simeq Ext^n_{G_r T}(K, M_1^*\otimes M_2).$$
We leave it for the reader to verify that the $G_r T$-supermodule $M_1^*\otimes M_2$ is isomorphic to the $G_r T$-supermodule
$(M_2^{<t>})^*\otimes M_1^{<t>}$. 
\begin{lm}\label{asinjanforsteinberg}
Every Steinberg supermodule is both injective and projective as a $G_r$-supermodule and as a $G_r T$-supermodule.
\end{lm}
\begin{proof}
Combining the above observation with Lemma \ref{asinjan}, one can easily superize Proposition II.10.2 from \cite{jan}.
\end{proof}
\begin{lm}\label{simplelemma}
Let $H$ be an algebraic supergroup and $N$ be a normal subsupergroup of $H$. For every $H$-supermodules $M_1$ and $M_2$ such that
$M_2^N=M_2$ there is an isomorphism of $H/N$-supermodules $Ext^{\bullet}_N(M_1, M_2)\simeq Ext^{\bullet}_N(M_1, K)\otimes M_2$. 
\end{lm}
\begin{proof}
Let $K\to I^{\bullet}$ be an injective resolution of the trivial $H$-supermodule. Then $M_2\to I^{\bullet}\otimes M_2$ is an injective resolution of $M_2$ as a $H$-supermodule and hence, also as a $N$-supermodule. Since 
$Hom_N(M_1, I^{\bullet}\otimes M_2)$ and $Hom_N(M_1, I^{\bullet})\otimes M_2$ are isomorphic to each other as complexes of $H/N$-supermodules, the claim follows.
\end{proof}
\begin{lm}\label{strangelemma}
If $H^0((p^r-1)\rho_0+\rho_{s, t})$ is a Steinberg supermodule, then for every $G_r$-supermodule $M$ we have  
$Ext^1_U (M, H^0_{ev}((p^r-1)\rho_0+\rho_{s, t}))^{G_{ev, r}}=0$.
\end{lm}
\begin{proof}
Denote the weight $(p^r-1)\rho_0+\rho_{s, t}$ by $\pi$, one more time. Observe that 
$H^0(\pi)|_{G_r}\simeq Z'_r(\pi)\simeq ind^{G_r}_{P_r} H^0_{ev}(\pi)$, where $M'=H^0_{ev}(\pi)|_{G_{ev, r}}$ is regarded as a $P_r$-supermodule via the epimorphism $P_r\to G_{ev, r}$.

As it has been already observed, $P_r$ is a faithfully exact subsupergroup of $G_r$ and the functor $ind^{G_r}_{P_r}$ is faithfully exact. 
Thus for every $G_r$-supermodule $M$ there is a natural isomorphism
$$
Ext^1_{G_r} (M, H^0(\pi))\simeq 
Ext^1_{P_r}(M, M').
$$
The functor $Hom_{P_r}(M, ?)$ is isomorphic to $Hom_U(M, ?)^{P_r/U}$. Since the functor $Hom_U(M, ?)$ from $P_r-\mathsf{SMod}$ to
$P_r/U-\mathsf{SMod}$ is left exact and takes injectives to injectives, there is a spectral sequence
$$E^{n, m}_2 =H^n (P_r/U, Ext^m_U(M, N))\Rightarrow Ext^{n+m}_{P_r}(M, N), N\in P_r-\mathsf{SMod}.
$$
Lemma \ref{simplelemma} combined with Proposition II.10.2 of \cite{jan} imply that
$Ext^m_U(M, M')\simeq Ext^m_U(M, K)\otimes M'$ is an injective $G_{ev, r}=P_r/U$-supermodule. In particular, $E^{n, m}_2 =0$ for every
$n\geq 1, m\geq 0$. Using the five-term exact sequence we obtain
$$Ext^1_{P_r} (M, M')\simeq Ext^1_U(M, M')^{P_r/U}.$$
Lemma \ref{asinjanforsteinberg} concludes the proof.
\end{proof}
Let $I_{ev}(\lambda)$ be the injective envelope of an irreducible  $G_{res}$-module of the highest weight $\lambda$.
We have already seen that $I_{ev}(\lambda)$ is also an injective envelope of the purely even irreducible $G_{ev}$-supermodule of the highest weight $\lambda$.
Denote by $I(\lambda)$ the injective envelope of the simple $G$-module of the highest weight $\lambda$.
\begin{tr}\label{tensorproductisinjective}
If $H^0((p^r-1)\rho_0+\rho_{s, t})$ is a Steinberg supermodule, then 
$$H^0((p^r-1)\rho_0+\rho_{s, t})\otimes I_{ev}(\lambda)^{[r]}\simeq I((p^r-1)\rho_0+\rho_{s, t}+p^r\lambda).$$
\end{tr}
\begin{proof}
We have again $G_{ev}=P_{ev}$, which implies that $G/P$ is an affine superscheme and the functor $ind^G_P ?$ is (faithfully) exact.
For every $\mu\in X(T)^+$ we have 
$$
Ext^1_G (L(\mu), H^0(\pi)\otimes I_{ev}(\lambda)^{[r]})\simeq 
Ext^1_P(L(\mu), H^0_{ev}(\pi)\otimes I_{ev}(\lambda)^{[r]}),
$$
where $\pi=(p^r-1)\rho_0+\rho_{s, t}$. Since $$H^0_{ev}(\pi)\otimes I_{ev}(\lambda)^{[r]}\simeq
I_{ev}(\pi+p^r\lambda)$$ (see  II.10.5.2 (1) of \cite{jan}), arguing as in Lemma \ref{strangelemma} we obtain
$$Ext^1_P(L(\mu), H^0_{ev}(\pi)\otimes I_{ev}(\lambda)^{[r]})\simeq Ext^1_U(L(\mu), H^0_{ev}(\pi)\otimes I_{ev}(\lambda)^{[r]}))^{P/U}\simeq$$
$$
(Ext^1_U(L(\mu), K)\otimes H^0_{ev}(\pi)\otimes I_{ev}(\lambda)^{[r]} )^{G_{ev}}\subseteq (Ext^1_U(L(\mu), K)\otimes H^0_{ev}(\pi)\otimes I_{ev}(\lambda)^{[r]} )^{G_{ev, r}}
$$
$$
=(Ext^1_U(L(\mu), H^0_{ev}(\pi))^{G_{ev, r}})^{\bigoplus\dim I_{ev}(\lambda)^{[r]}}=0.
$$
Therefore $H^0(\pi)\otimes I_{ev}(\lambda)^{[r]}$ is an injective $G$-supermodule.
As in the proof of Proposition \ref{asdonkin}, we obtain that $H^0(\pi)\otimes I_{ev}(\lambda)^{[r]}$ is indecomposable. Since this supermodule
contains a subsupermodule $H^0(\pi+p^r\lambda)$, its socle coincides with $L(\pi+p^r\lambda)$.
\end{proof}
\begin{rem}\label{aboutblocks}
Let $(p^r-1)\rho_0 +\rho_{s, t}$ be a highest weight of some Steinberg supermodule. Define a map $\theta_r : X(T)^+\to X(T)^+$ by
$\theta_r(\lambda)= (p^r-1)\rho_0 +\rho_{s, t}+p^r\lambda$. 
Theorem \ref{tensorproductisinjective} implies that if $\mathfrak{B}$ is a block of $GL(m|n)$-supermodules, then $\theta_r(\mathfrak{B})$ is a block as well. The proof can be modified from the proof of Corollary 2.3, \cite{don3}. 
\end{rem}

\section{Examples}

Consider two $G$-supermodules $M$ and $M'$ over an algebraic supergroup $G$. The right  $K[G]$-supecormodule structure on $M\otimes M'$ is given by
$\tau_{M\otimes M'}(m\otimes m')= \sum (-1)^{|f_2||m_1|}m_1\otimes m'_1\otimes f_2 g_2$, where
$\tau_M(m)=\sum m_1\otimes f_2$ and $\tau_{M'}(m')=\sum m'_1\otimes g_2$. The following lemma is evident.
\begin{lm}\label{distaction}
$Dist(G)$ acts on $M\otimes M'$ by the rule
$$\phi\cdot (m\otimes m')=\sum (-1)^{|f_2||m_1|+|\phi|(|m_1|+|m'_1|)+|\phi_2||f_2|} m_1\otimes m'_1\phi_1(f_2)\phi_2(g_2). 
$$
In particular, if $\phi$ is a primitive element, then 
$$\phi\cdot (m\otimes m')=\phi\cdot m\otimes m' + (-1)^{|\phi||m|}m\otimes \phi\cdot m'.$$ 
\end{lm} 
As above, let $G=GL(m|n)$. From now on we assume that $m, n\geq 1$. 

The superalgebra $Dist(U^{opp})$ is generated by the primitive elements (matrix units) $e_{ij}$ for $1\leq i\leq m< j\leq m+n$. 
They act on basis vectors $w_k$ of $W$ as $e_{ij}\cdot w_k=\delta_{jk}w_i$ for $1\leq k\leq m+n$.

Consider the $k$-th (super) exterior power $\Lambda^k(W)=\bigoplus_{0\leq i\leq \min\{m, k\}}\Lambda^i(W_0)\otimes S^{k-i}(W_1)$.
Basis vectors of $\Lambda^k(W)$ are of the form 
$$
w_{i_1}\ldots w_{i_s} w_{m+1}^{\beta_1}\ldots w_{m+n}^{\beta_n}, \mbox{ where } 1\leq i_1 <\ldots < i_s\leq m \mbox{ and } \beta_1, \ldots , \beta_n\geq 0.
$$
We will denote $w_{i_1}\ldots w_{i_s}$ by $w_+^I$, where $I=\{i_1, \ldots , i_s\}$, and $w_{m+1}^{\beta_1}\ldots w_{m+n}^{\beta_n}$ by $w_-^{\beta}$, where 
$\beta=(\beta_1 ,\ldots , \beta_n)$. Also denote $\sum_{1\leq t\leq n}\beta_t$ by $|\beta|$.
\begin{lm}\label{exterior} There is a decomposition
$$\Lambda^k(W)^{U^{opp}}=(\bigoplus_{0\leq i\leq \min\{k,m-1 \}, p|(k-i)}\Lambda^i(W_0)\otimes S^{\frac{k-i}{p}}(W_1^p))\bigoplus \Lambda^m(W_0)\otimes S^{k-m}(W_1).$$ 
Here $S^{\frac{k-i}{p}}(W_1^p)$ is a subspace of $S^k(W_1)$ spanned by the elements 
$w_-^{p\gamma}$, where $|\gamma|=\frac{k-i}{p}$. The last summand appears only if $k\geq m$.
\end{lm}
\begin{proof}
If $1\leq i\leq m < j\leq m+n$, then
$$e_{ij}\cdot (w_+^I w_-^{\beta})=\beta_{j-m} w_+^I w_i w_{m+1}^{\beta_1}\ldots w_j^{\beta_{j-m}-1}\ldots w_{m+n}^{\beta_n}.
$$
The decomposition in the statement of the lemma follows.
\end{proof}

\begin{cor}\label{whengood}
The $G_{ev}$-supermodule $\Lambda^k(W)^{U^{opp}}$ has a good filtration if and only if $n=1$ or $k=ps + r$, where
$s\geq 0$ and $m\leq r<p$.
\end{cor}  
\begin{lm}\label{n=1}
If $n=1$, then $\Lambda^k(W)$ has a good filtration if and only if $k=ps+r$,
where $s\geq 0$ and $m\leq r < p$.
In this case $\Lambda^k(W)\simeq H^0((1^m|k-m)^a)=L((1^m|k-m)^a)=V((1^m|k-m)^a)$, where
$a=k-m \pmod 2$, is a tilting supermodule.
\end{lm}
\begin{proof}
Assume that $\Lambda^k(W)$ has a good filtration.
By Lemma \ref{injandproj}, $\dim\Lambda^k(W)$ is a multiple of $\dim K[U^{opp}]=2^m$.
Thus $k\geq m$, $\dim\Lambda^k(W)=\dim\Lambda^m(W)=2^m$ and $\Lambda^k(W)|_{U^{opp}}\simeq \Pi^a K[U^{opp}]$ for $a=0,1$. 

Since $K[U^{opp}]^{U^{opp}}=K$ and $\Lambda^k(V)$ has a unique highest weight $(1^m|k-m)$, Lemma \ref{exterior} and Remark \ref{asinjective} imply that 
$k=ps+r$, where $s\geq 0, m\leq r < p$ and 
$\Lambda^k(W)\simeq H^0((1^m| k-m)^a)$ for $a= k-m \pmod 2$. Moreover, by Theorem 1 of \cite{marko}, $H^0((1^m|k-m))$
is irreducible, hence $H^0((1^m|k-m)^a)=L((1^m|k-m)^a)=V((1^m|k-m)^a)$ is tilting. In fact, $((1^m|k-m)+\rho, \epsilon_i-\epsilon_{m+1})=k+1-i$ is not divided by $p$ for any $i$. 

Conversely, let $L(\pi^b)$ be a simple subsupermodule of $\Lambda^k(W)$, where 
$k=ps+r, s\geq 0$ and $m\leq r < p$. Then its highest weight vector is annihilated 
by $Dist(B^{opp})^+$. Since there is only one such vector  $w_1\ldots w_m w_{m+1}^{k-m}$ of weight $\pi=(1^m|k-m)$, we get
$L((1^m|k-m)^a)\subseteq\Lambda^m(W)$. As it has already been observed, $L((1^m|k-m))=H^0((1^m|k-m))$
and this supermodule has the dimension $2^m$. Thus $\Lambda^k(W)\simeq H^0((1^m|k-m)^a)$.
\end{proof}
Since $U^{opp}$ is abelian, any subsupergroup of $U^{opp}$ is normal. 
Let $N$ be a subsupergroup of $U^{opp}$ that is defined by the equations $c_{i, m+1}=0$ for $1\leq i\leq m$ . In other words,
an element $g\in U^{opp}$ belongs to $N$ if and only if $gw_{m+1}=w_{m+1}$. Let $W'$ denote the subsuperspace $\sum_{1\leq i\leq m+n, i\neq m+1}Kw_i$ of $W$.
Then
$$\Lambda^k(W)|_N\simeq \bigoplus_{0\leq t\leq k }\Lambda^{k-t}(W').$$
If $n>1$ and $\Lambda^k(W)$ is an injective $U^{opp}$-supermodule, 
then every $\Lambda^s(W')$ for $0\leq s\leq k,$ is injective as a $N$-supermodule by Lemma \ref{reduction}. On the other hand,
every (finite-dimensional) injective $N$-supermodule has a dimension that is divided by $2^{m(n-1)}\geq 2$. Thus $\Lambda^0(W')=K$ can not be injective as a $N$-supermodule. This contradiction shows that the only case when  
$\Lambda^k(W)$ is injective as a $U^{opp}$-supermodule is $n=1$.
The following proposition is now evident.
\begin{pr}\label{decreasingof}
$\Lambda^k(W)$ has a good filtration if and only if $n=1$ and $k=ps+r$,
where $s\geq 0$ and $m\leq r < p$. 
\end{pr}
In the previous notation, $e_{ji}\cdot (w_+^I w_-^{\beta})\neq 0$ if and only if $i\in I$. In this case 
$e_{ji}\cdot (w_+^I w_-^{\beta})=\pm w_+^{I\setminus i} w_-^{\beta'}$, where $\beta'_j=\beta_j +1$ and 
$\beta'_k=\beta_k$ for $k\neq j$. Thus $Dist(U)^+\cdot\Lambda^k(W)$ is spanned by all $w_+^I w_-^{\beta}$ with
$I^{\sharp} < m$ and $|\beta|\geq 1$. Therefore, 
$\Lambda^k(W)/\Lambda^k(W)_U\simeq\Lambda^m(W_0)\otimes S^{k-m}(W_1)$ provided $k > m$, and 
$\Lambda^k(W)/\Lambda^k(W)_U\simeq \Lambda^k(W_0)$ otherwise.
\begin{pr}\label{Weylfiltrforfactor}
The $G$-supermodule $\Lambda^k(W)$ has a Weyl filtration if and only if $n=1$ and $k=ps+r$,
where $s\geq 0$ and $m\leq r < p$.
\end{pr}
\begin{proof}
Assume that $\Lambda^k(W)$ has a Weyl filtration. Then $\Lambda^k(W)$ is a projective $U$-supermodule, hence injective by Lemma \ref{projandinjoverinfinitesimal}. Let $N$ be a 
subsupergroup of $U$ that is defined by the equations $c_{m+n, i}=0$ for $1\leq i\leq m$. Then
$$\Lambda^k(W)|_N\simeq \bigoplus_{0\leq t\leq k}\Lambda^{k-t}(W'),$$
where $W'=\sum_{1\leq i\leq m+n-1} Kw_i$. Arguing as before, $n=1$ and $k\geq m$ is the only case when $\Lambda^k(W)$ can be an injective, hence projective $U$-supermodule. 
Remark \ref{proj} implies that $\Lambda^k(W)$ has to be isomorphic to
$V((1^m| k-m)^a)$ for $a=k-m\pmod 2$. Furthermore, since the vector $w_1\ldots w_m w_{m+1}^{k-m}$ is annihilated by $Dist(B^{opp})^+$, it implies that
$L((1^m| k-m)^a)\subseteq \Lambda^k(W)\simeq V((1^m| k-m)^a)$. On the other hand, $L((1^m| k-m)^a)$ is the top of $V((1^m| k-m)^a)$.
Since $V((1^m| k-m)^a)$ is indecomposable, $L((1^m| k-m)^a)=V((1^m| k-m)^a)=H^0((1^m| k-m)^a)$. Theorem 1 from \cite{marko}
implies that $k=ps+r$, where $s\geq 0$ and $m\leq r < p$. Lemma \ref{n=1} concludes the proof.
\end{proof}

Any basis element of $S^k(W)$ has the form 
$w_+^{\gamma}w_-^J=w_1^{\gamma_1}\ldots w_m^{\gamma_m}
w_{j_1}\ldots w_{j_s}$, where $\gamma_1, \ldots,\gamma_m\geq 0$ and $m+1\leq j_1 <\ldots < j_s\leq m+n$. 
As we have seen earlier,
$
e_{ij}\cdot (w_+^{\gamma}
w_-^J)\neq 0$ if and only if there is an index $t$ such that $j=j_t$. In this case 
$$e_{ij}\cdot (w_+^{\gamma}
w_-^J)=(-1)^{t-1}
w_1^{\gamma_1}\ldots w_i^{\gamma_i+1}\ldots w_m^{\gamma_m}
w_-^{J\setminus \{j_t\}}.
$$
Analogously, we have
$$
e_{ji}\cdot (w_+^{\gamma}
w_-^J)=\gamma_i w_1^{\gamma_1}\ldots w_i^{\gamma_i-1}\ldots w_m^{\gamma_m}
w_j w_-^J.
$$
In particular, $S^k(W)^{U^{opp}}=S^k(W_0)=H_{ev}^0((k| 0^n))$. 

If $N=Stab_{U^{opp}}(W')$, where
$W'=\sum_{1\leq i\leq m+n, i\neq m} Kw_i$, then
$$S^k(W)|_N\simeq\bigoplus_{0\leq t\leq k} S^{k-t}(W').
$$
Again, if $S^k(W)$ is an injective $U^{opp}$-supermodule, then $m=1$ and $k\geq n$.
Analogously, if $S^k(W)$ is a projective $U$-supermodule, then it is also injective. 
Working with the subsupergroup of $U$ defined by the equations $c_{jm}=0$ for $m+1\leq j\leq m+n$,
we derive that this is possible only if $m=1$ and $k\geq n$.
\begin{pr}\label{symmasofgood}
$S^k(W)$ has a good filtration if and only if $m=1$ and $n\leq k < p$. In this case
$S^k(W)= H^0((k|0^n))=L((k|0^n))=V((k|0^n))$ is tilting.
\end{pr}
\begin{proof}
If $S^k(W)$ has a good filtration, then $m=1$ and $k\geq n$. Remark \ref{asinjective} implies that
$S^k(W)\simeq H^0((k|0^n))$. Let $k=ps +r$, where $s\geq 0$ and $0\leq r < p$. Assume that 
$s > 0$. Then the subsuperspace 
$$W'=\sum_{J^{\sharp}=r-t, 0\leq t\leq r} Kw_1^{ps+t}w_-^{J}$$ is a proper subsupermodule in $S^k(W)$. 
But the highest weight vector $w_1^k$ does not belong to $W'$, which is a contradiction. Thus $s=0$ and 
$H^0((k|0^n))$ is irreducible by Theorem 1 of \cite{marko}. Conversely,  
the vector $w_1^k$ of the unique highest weight $(k|0^n)$ generates a simple subsupermodule of $S^k(W)$
that is isomorphic to $L((k|0^n))$.
If $m=1$ and $n\leq k < p$, then $H^0((k|0^n))=L((k|0^n))$ has the dimension $2^n$, and the proof is concluded.
\end{proof}
\begin{pr}\label{weylofsymm}
$S^k(W)$ has a Weyl filtration if and only if $m=1$ and $n\leq k< p$. 
\end{pr}
\begin{proof}
The conditions $m=1$ and $k\geq n$ are necessary for $S^k(W)$ to have a Weyl filtration.
Furthermore, we have 
$$S^k(W)_U=\sum_{p\not|(t+1), 0\leq t\leq k-1, k=t+J^{\sharp}} Kw_1^t w_-^J .$$
This implies
$$
S^k(W)/S^k(W)_U\simeq S^k(W_0)\bigoplus (\sum_{1\leq l\leq [\frac{k+1}{p}], pl-1+J^{\sharp}=k, J^{\sharp}\geq 1} Kw_1^{pl-1}w_-^J).
$$
As above, $S^k(W)|_U\simeq \Pi^b Dist(U)$, where $b=0,1$. Therefore $\dim S^k(W)/S^k(W)_U=1$, hence $k < p$ and $S^k(W)\simeq V((k|0^n))=L((k|0^n))=H^0((k|0^n))$ is tilting. The converse statement follows by Proposition \ref{symmasofgood}. 
\end{proof}

\end{document}